\documentclass[11pt,reqno]{amsart}
\usepackage{amsmath,amsthm,amscd,amsfonts,amssymb,color}
\usepackage{cite}
\usepackage{graphicx}
\usepackage{xcolor}
\usepackage[mathscr]{eucal}
\usepackage{upgreek}
\usepackage[bookmarksnumbered,colorlinks,plainpages]{hyperref}
\newtheorem{theorem}{Theorem}[section]

\newtheorem*{Acknowledgement}{\textnormal{\textbf{Acknowledgement}}}

\theoremstyle{definition}

\newtheorem{definition}[theorem]{Definition}
\newtheorem{example}[theorem]{Example}

\newtheorem{Open Prob}[theorem]{Open Problem}
\theoremstyle{remark}
\numberwithin{equation}{section}
\def\DJ{\leavevmode\setbox0=\hbox{D}\kern0pt\rlap{\kern.04em\raise.188\ht0\hbox{-}}D}
\begin{document}

\title[On A Novel Approach to Nonexpansive Mappings]{ On A Novel Approach to Nonexpansive Mappings }

\author[A.\ Banerjee, H.\ Garai, P.\ Mondal, L.K.\ Dey]
{Anish Banerjee$^{1}$, Hiranmoy Garai$^{2}$, Pratikshan Mondal$^{3}$, Lakshmi Kanta Dey$^{1}$}

\address{{$^{1}$}   Anish Banerjee,
                    Department of Mathematics,
                    National Institute of Technology
                    Durgapur, West Bengal,  India}
                    \email{anish22.mathematics@gmail.com}
\address{{$^{2}$}   Hiranmoy Garai,
                    Department of Mathematics,
                    Government General Degree College Ranibandh, 
                    Bankura, West Bengal, India}
                    \email{hiran.garai24@gmail.com} 
\address{{$^{3}$}   Pratikshan Mondal,
                    Department of Mathematics,
                    A.B.N. Seal College, 
                    Cooch Behar, West Bengal, India}
                    \email{real.analysis77@gmail.com}                   
\address{{$^{1}$}   Lakshmi Kanta Dey,
                    Department of Mathematics,
                    National Institute of Technology
                    Durgapur, West Bengal, India}
                    \email{lakshmikdey@yahoo.co.in}


\begin{abstract}
This paper seeks to advance the theory of nonexpansive mappings by introducing and exploring a novel class of nonexpansive type mappings, which we aptly designate as perimetric nonexpansive mappings. We establish that the collection of mappings we propose is considerably larger than the existing classes of nonexpansive and quasi-nonexpansive mappings. We also establish fixed point existence findings by examining the connection between periodic points and fixed points in the context of normed linear spaces. Finally, we establish a significant result by proving that every perimetric nonexpansive mapping on a closed bounded convex subset of a Hilbert space necessarily has a fixed point.
\end{abstract}

\maketitle
\setcounter{page}{1}

		
\section{Introduction}
\noindent
The concept of contraction mappings was introduced by Banach \cite{B22} in 1922, and the notion of nonexpansive mappings was developed independently by Browder \cite{B65}, Göhde \cite{G65}, and Kirk \cite{K65} in 1965.  In a metric space $(X,d)$, a mapping $T:X\to X$ is called a contraction  if for a constant $\alpha\in (0,1]$, $d(Tx,Ty)\leq \alpha d(x,y)$ holds for all $x,y\in X$ and $T$ is called nonexpansive if $d(Tx,Ty)\leq  d(x,y)$ holds for all $x,y\in X.$
Subsequently, a significant amount of work has been developed by mathematicians to extend and generalize the notions of contraction and nonexpansive mappings and these developments drew upon two key mathematical research areas,   namely, fixed point theory of contractions type mappings and fixed point theory of nonexpansive type mappings. Fixed point theory of nonexpansive type mappings have been explored in two main directions: one by generalizing contraction type mappings to nonexpansive type mappings and the other by directly extending and generalizing nonexpansive mappings themselves. Of these two directions, the first one holds particular significance since nonexpansive mappings were developed as a generalization of Banach's contraction mapping.  As a result, many nonexpansive type mappings have been introduced as generalizations of their contraction counterparts, with most research focused on normed linear spaces.

Recently in $2023$, Petrov \cite{P23}  presented a new class of contraction type mapping precisely defined as: 
\begin{definition}\cite[\, p. 2, Definition~2.1]{P23}\label{D11}
Let $(M,d)$ be a metric space with at least three points. Then  $T: M \to M$ is said to be a mapping contracting perimeters of triangles on $M$ if there exists $\alpha \in [0,1)$ such that the inequality
    \begin{equation}\label{mcpt1}
        d(Tx,Ty)+ d(Ty,Tz)+ d(Tz,Tx) \le \alpha (d(x,y)+ d(y,z)+ d(z,x))
        \end{equation}
        holds for all three pairwise distinct points $x, y, z \in M$.
\end{definition}
Petrov established an existence result for fixed points for this class of mappings, which can be stated as follows:

\begin{theorem}\cite[\, p. 3, Theorem~2.4]{P23}\label{T12}
Let $(M,d)$, $|M| \geq 3$, be a complete metric space and $T:X\to X$ be a mapping contracting perimeters of triangle on $M$. Then $T$ has a fixed point if and only if $T$ does not possess periodic points of prime period $2$. The number of fixed points is at most two.
\end{theorem}
Bey et al. \cite{BPS25} presented the contractive version of perimetric contractions, which is formulated as follows:

\begin{definition}\cite[\, p. 7, Definition~3.1]{BPS25}\label{D13}
Let $(M,d)$ be a metric space with at least three points. Then  $T: M \to M$ is said to be a mapping contracting perimeters of triangles in the sense of Edelstein on $M$ if the inequality
    \begin{equation}\label{mcpte}
        d(Tx,Ty)+ d(Ty,Tz)+ d(Tz,Tx) < (d(x,y)+ d(y,z)+ d(z,x))
        \end{equation}
        holds for all three pairwise distinct points $x, y, z \in M$.
\end{definition}
Furthermore, Bey et al. established an existence result for the aforementioned mappings, which is presented below:
\begin{theorem}\cite[\, p. 7, Theorem 3.3]{BPS25}\label{T14}
Let $X$ be a metric space and let $T: X \to X$ be a mapping contracting perimeters of triangles in the sense of Edelstein. If $T$ does not possess periodic points of prime period $2$ and there exists an $x \in X$ such that the sequence $\{T^n x\}$ has a convergent subsequence which converges to $z \in X$, then $z$ is a fixed point of $T$. The number of fixed points of $T$ is at most two.
\end{theorem}

This new class of contraction type mappings has been extensively studied by some mathematicians in diverse directions, as seen in \cite{ PB24, PP24,  BMD25, BMD24, P25}. However, these studies have been restricted to the fixed point theory of contraction type mappings only. Given the current importance of fixed point theory for nonexpansive type mappings, it's natural to wonder if a nonexpansive counterpart of this contraction type mapping can be developed. Building on these insights,  in this article, we develop the nonexpansive analogue of the mapping contracting perimeters of triangles, which we designate as perimetric nonexpansive mappings. As the majority of nonexpansive type mapping research has centered on normed linear spaces, we develop perimetric nonexpansive mappings within this framework, rather than in the more general context of metric spaces. Furthermore, we develop key properties concerning the existence of fixed points for these new perimetric nonexpansive mappings in the context of Banach and Hilbert spaces. More precisely, we prove fixed point existence theorems by investigating the relationship between periodic and fixed points of mappings in Banach spaces and additionally, we demonstrate that every perimetric nonexpansive mapping defined on a closed bounded convex subset of a Hilbert space has a fixed point.

\section{Perimetric nonexpansive mappings}
We begin this section by giving the formal definition of perimetric nonexpansive mappings.
\begin{definition}\label{D21}
Let $(X,d)$ be a normed linear space. A mapping $T: X \to X$ is said to be a perimetric nonexpansive mapping if  the following holds for all distinct points $x,y,z \in X$:
\begin{align}
\lVert Tx-Ty \rVert+ \lVert Ty-Tz \rVert+ \lVert Tz-Tx \rVert
\le \lVert x-y \rVert+ \lVert y-z \rVert+ \lVert z-x \rVert. \label{pnm}
\end{align}
\end{definition}
It follows from the definition that every nonexpansive mapping is a perimetric nonexpansive mapping, but the converse implication does not hold. This statement is supported by the subsequent example.

\begin{example}\label{E22}
Let us consider the normed linear space $(\mathbb{R}^2,\|\cdot\|)$  equipped with the norm ${\lVert (u,v) \rVert}_1 = |u| + |v|$ and consider  a subset $K=\{(0,0), (1,0), (0,1)\} \cup \{\alpha (1,1): \alpha \geq 1\}$ of $\mathbb{R}^2$ . We define a mapping $T:K \to K$ by 
\begin{align*}
T(u,v)= 
\begin{cases}
(u,v), &\text{if $(u,v)\in \{(0,0)\} \cup \{\alpha (1,1): \alpha \geq 1\}$};\\ 
(1,1), &\text{if $(u,v)\in \{(1,0), (0,1)\}$}.
\end{cases}
\end{align*}

\medskip

For any $x, y, z \in K$, consider the following cases:

\medskip

Case 1: Let $x=(0,0), y=(1,0), z=(0,1)$.
Then
\begin{align*}
  {\lVert Tx-Ty \rVert}_1 + {\lVert Ty-Tz \rVert}_1 + {\lVert Tz-Tx \rVert}_1 = 4 ={\lVert x-y \rVert}_1 + {\lVert y-z \rVert}_1 + {\lVert z-x \rVert}_1.  
\end{align*}

\medskip

Case 2: Let $x=(0,0), y=(1,0), z=\alpha(1,1)$.
Then
\begin{align*}
  {\lVert Tx-Ty \rVert}_1 + {\lVert Ty-Tz \rVert}_1 + {\lVert Tz-Tx \rVert}_1 
  = 4 \alpha ={\lVert x-y \rVert}_1 + {\lVert y-z \rVert}_1 + {\lVert z-x \rVert}_1. 
\end{align*}

\medskip

Case 3: Let $x=(0,0), y=(0,1), z=\alpha(1,1)$.
Then
\begin{align*}
  {\lVert Tx-Ty \rVert}_1 + {\lVert Ty-Tz \rVert}_1 + {\lVert Tz-Tx \rVert}_1 
  = 4 \alpha ={\lVert x-y \rVert}_1 + {\lVert y-z \rVert}_1 + {\lVert z-x \rVert}_1.  
\end{align*}

\medskip

Case 4: Let $x=(1,0), y=(0,1), z=\alpha(1,1)$.
Then
\begin{align*}
  &{\lVert Tx-Ty \rVert}_1 + {\lVert Ty-Tz \rVert}_1 + {\lVert Tz-Tx \rVert}_1 \\
 & = 4(\alpha -1)\\
  &< 4 \alpha={\lVert x-y \rVert}_1 + {\lVert y-z \rVert}_1 + {\lVert z-x \rVert}_1.
\end{align*}

\medskip

Case 5: Let $x=(0,0), y=\alpha_1(1,1), z=\alpha_2(1,1)$ such that $\alpha_2 < \alpha_1$.
Then
\begin{align*}
  {\lVert Tx-Ty \rVert}_1 + {\lVert Ty-Tz \rVert}_1 + {\lVert Tz-Tx \rVert}_1 = 4 \alpha_1 ={\lVert x-y \rVert}_1 + {\lVert y-z \rVert}_1 + {\lVert z-x \rVert}_1. 
\end{align*} 

\medskip

Case 6: Let $x=(0,1), y=\alpha_1(1,1), z=\alpha_2(1,1)$ such that $\alpha_2 < \alpha_1$. 
Then
\begin{align*}
  &{\lVert Tx-Ty \rVert}_1 + {\lVert Ty-Tz \rVert}_1 + {\lVert Tz-Tx \rVert}_1 \\
  &= 4 \alpha_1 - 4\\
  &<4 \alpha_1 - 2={\lVert x-y \rVert}_1 + {\lVert y-z \rVert}_1 + {\lVert z-x \rVert}_1.
\end{align*}

\medskip

Case 7: Let $x=(1,0), y=\alpha_1(1,1), z=\alpha_2(1,1)$ such that $\alpha_2 < \alpha_1$. This case is analogous to the previous one.

\medskip

Case 8: Let $x=\alpha_1(1,1), y=\alpha_2(1,1), z=\alpha_3(1,1)$ such that $\alpha_3 < \alpha_2 < \alpha_1$.
Then
\begin{align*}
  &{\lVert Tx-Ty \rVert}_1 + {\lVert Ty-Tz \rVert}_1 + {\lVert Tz-Tx \rVert}_1\\ 
  &= 4 \alpha_1 - 4 \alpha_3\\
  &={\lVert x-y \rVert}_1 + {\lVert y-z \rVert}_1 + {\lVert z-x \rVert}_1.
\end{align*}

\medskip

Thus, ${\lVert Tx-Ty \rVert}_1 + {\lVert Ty-Tz \rVert}_1 + {\lVert Tz-Tx \rVert}_1 \leq {\lVert x-y \rVert}_1 + {\lVert y-z \rVert}_1 + {\lVert z-x \rVert}_1$ for all distinct $x,y,z \in K$. Therefore, $T$ is a perimetric nonexpansive mapping. Yet, $T$ fails to be both nonexpansive and quasi-nonexpansive, since for $x=(1,0)$ and $y=(0,0)$, we have ${\lVert Tx-Ty \rVert}_1 = 2$ and ${\lVert x-y \rVert}_1 =1$ and $y$ is a fixed point of $T$. 
\end{example}

The mapping $T$ in the above example has a fixed point, and indeed, it has infinitely many fixed points. However, this may not be true in general, as illustrated by the following examples:
\begin{example}\label{E23}
Let $X= \mathbb{R}$ be a normed linear space equipped with the norm ${\lVert x \rVert}= |x|$, for all $x \in X$. Then the mapping $T: X \to X$ defined by $Tx=x+1$, for all $x \in X$,
 is a perimetric nonexpansive mapping without having a fixed point.
\end{example}

\begin{example}\label{E24}
Let $X= \mathbb{R}^2$ be a normed linear space equipped with the norm ${\lVert (x,y) \rVert}_1 = |x| + |y|$, for all $(x,y) \in X$ and consider a subset $K=[0,1] \times [0,1]$ of $X$. We define a mapping $T:K \to K$ by $T(x,y)= (1-x,1-y)$, for all $(x,y) \in K$. Then $T$ is a perimetric nonexpansive mapping with $(\frac{1}{2}, \frac{1}{2})$ as the unique fixed point of $T$.
\end{example}

\section{Main Results}
In this section, we investigate key properties of perimetric nonexpansive mappings, advancing the theoretical foundations of these mappings.

It is known that some of the existing nonexpansive type mappings are continuous, while the others are not so. Here we prove that perimetric nonexpansive mappings are always continuous.

\begin{theorem}\label{T31}
A perimetric nonexpansive mapping $T$ on a normed linear space $(X,\|\cdot\|)$ is continuous.
\end{theorem}

\begin{proof}
Let $x^*\in X$ be arbitrary. If $x^*$ is an isolated point of $X$, then it is obvious that $T$ is continuous at $x^*$. 

Now, suppose that $x^*$ is a limit point of $X$. Then for any $\varepsilon>0$, we choose $ \delta>0 $ such that $0<\delta<\frac{\varepsilon}{4}$. 

Since $x^*$ is a limit point of $X$, there exist $y \in X$ with $ x \neq y$ such that $\lVert x^* -y \rVert < \delta$. Then, for all $x \in X$ with $x \neq x^*$ satisfying $\lVert x-x^* \rVert < \delta$, we have
\begin{align*}
\lVert Tx-Tx^* \rVert &\le \lVert Tx-Tx^* \rVert + \lVert Tx^*-Ty \rVert + \lVert Ty-Tx \rVert\\
&\le \lVert x-x^* \rVert + \lVert x^*-y \rVert + \lVert y-x \rVert\\
&\le 2( \lVert x-x^* \rVert + \lVert x^*-y \rVert)\\
& < 4 \delta < \varepsilon.
\end{align*}
Therefore, the conclusion follows.
\end{proof}
The continuity of perimetric nonexpansive mappings allows us to identify a key characteristic of $F(T)$, the set of all fixed points of such mappings.

\begin{theorem}\label{T32}
$F(T)$, the set of all fixed points of a perimetric nonexpansive mapping $T$ on a normed linear space $(X,\|\cdot\|)$ is closed.
\end{theorem}
\begin{proof}
Let $z$ be any adherent point of $F(T)$. Then we get a sequence $\{z_n\}$ in $F(T)$ converging to $z$. Since $T$ is continuous and $Tz_n=z_n$ for all $n$, it follows that $z=Tz$, i.e., $z\in F(T)$. So, $F(T)$ is closed.
\end{proof}

It is easy  to confirm that Theorem \eqref{T12} and Theorem \eqref{T14}  remain valid when the underlying structure is a normed linear space rather than a metric space. Based on these findings, we derive a fixed point existence theorem for perimetric nonexpansive mappings in Banach spaces.

\begin{theorem}\label{T33}
Let $(X,\|\cdot\|)$ be a Banach space and let $K$ be a compact convex subset of $X$ containing the null vector $\theta$. Let $T: K \to K$ be a perimetric nonexpansive mapping with no periodic point of prime period $2$. Then $T$ has a fixed point in $K$.
\end{theorem}

\begin{proof}
For each $n \in \mathbb{N}$, we define the mapping $T_n:K \to K$ by $T_n(x)=(1 - \frac{1}{n}) Tx$ for all $x \in K$. Then  for any distinct $x,y,z \in K$ and for all $n \in \mathbb{N}$, we have 
\begin{align*}
&\lVert T_n(x) -T_n(y)\rVert + \lVert T_n(y) -T_n(z)\rVert +\lVert T_n(z) -T_n(x)\rVert\\
=& \frac{n-1}{n} (\lVert Tx -Ty\rVert + \lVert Tx - Ty\rVert + \lVert Tx -Ty \rVert)\\
\le &\frac{n-1}{n} (\lVert x -y\rVert + \lVert x - y\rVert + \lVert x - y \rVert)  \hspace{2cm} [\text{using } (\ref{pnm})].
\end{align*}
Therefore, $T_n$ is a mapping contracting perimeters of triangles on $K$. 
If possible, suppose that for each $n \in \mathbb{N}$, $T_n$ has a periodic point $p_n \in K$ (say) with prime period $2$. Then we have
\begin{align*}
    T_n^2 (p_n)= p_n \text{ for all } n \in \mathbb{N}.
\end{align*}

Since $K$ is compact, $\{p_n\}$ has a subsequence $\{p_{n_i}\}$ converging to some $p \in K$ as $i \to \infty$. Since $T_n \to T$ uniformly and $T$ is continuous, it follows that $ {T_{n_i}} (p_{n_i}) \to Tp$ as $i \to \infty$. Again applying the uniform convergence of $\{T_{n}\}$, we get $ {T_{n_i}} (T_{n_i}p_{n_i}) \to T (Tp )$, which implies that $ p_{n_i} \to T^2 p$. Thus we have $T^2p = p$, i.e., $p$ is a periodic point of $T$ of prime period $2$,   a contradiction to our assumption.
Hence $T_n$ has no periodic point of prime period $2$ for each $n$. Therefore, by Theorem~\ref{T12}, $T_n$ has a fixed point $x_n$ in $K$, i.e., $T_n x_n=x_n$ for each $n$. As $K$ is compact, there exists a subsequence $\{x_{n_k}\}$ of $\{x_n\}$ which converges to $x^* \in K.$ Since, the sequence of functions $\{T_{n_k}\}$ converges uniformly to $T$ and $T$ is continuous, it follows that $T_{n_k} x_{n_k} \to T x^*$ as $k\to \infty$. On the other hand since $T_{n_k} x_{n_k} =x_{n_k}$, we have $T_{n_k} x_{n_k} \to  x^*$ as $k\to \infty$. Thus, $Tx^*=x^*$, i.e., $x^*$ is a fixed point of $T$.
\end{proof}

Next, we demonstrate that even if the underlying space is merely a normed linear space rather than a Banach space, $T$ can still have a fixed point.

\begin{theorem}\label{T34}
Let $(X,\|\cdot\|)$ be a normed linear space and $K$ be a compact convex subset of $X$ containing the null vector $\theta$. Let $T: K \to K$ be a perimetric nonexpansive mapping and $(I-T)K$ be a closed subset of $X$. If $T$ does not possess periodic points of prime period $2$ and if there exists an $x \in X$ such that the sequence $\{T^n x\}$ has a convergent subsequence which converges to $\xi \in X$, then $\xi$ is a fixed point of $T$. 
\end{theorem}

\begin{proof}
Consider a sequence $\{t_n\}$ of positive real numbers, where $0 < t_n < 1$ for all $n$, and $\lim_{n \to \infty} t_n = 1$. For each $n \in \mathbb{N}$, we define $T_n:K \to K$ by $T_n (x)=t_n Tx$ for all $x \in K$.

Then it is straightforward that for each $n\in \mathbb{N}$, $T_n$ is a mapping contracting perimeters of triangles in the sense of Edelstein.

Suppose that for each $n \in \mathbb{N}$, $T_n$ has a periodic point $p_n$ (say) with prime period $2$. Then, proceeding as in Theorem~\ref{T33}, we can arrive at a contradiction to the fact that $T$ has no periodic point of prime period $2$. Thus, $T_n$ has no periodic point of prime period $2$ for each $n$. Hence, by Theorem~\ref{T14}, $T_n$ admits a fixed point  $x_n \in K$. 

Since $K$ is compact,  it is bounded and so $\lVert Tx \rVert \le r$ for all $x \in K$ for some $r>0$.
For all $n \in \mathbb{N}$, we obtain 
\begin{align*}
&\lVert Tx_n - x_n \rVert = \lVert Tx_n - t_n Tx_n \rVert = \lVert (1-t_n) Tx_n \rVert \le r(1-t_n)\\
\implies &(I-T)(x_n) \to \theta \text{ as } n \to \infty\\
\implies &\theta \in (I-T)K~~ [\text{since $(I-T)K$ is closed} ] .
\end{align*}
Hence, there exists $\xi \in K$ such that $(I-T)(\xi)= \theta$, i.e., $T (\xi)=\xi$.
Hence $\xi$ is a fixed point of $T$.
\end{proof}

The above proof reveals that the compactness of the domain of $T$ is essential for guaranteeing the existence of fixed points. We now demonstrate that in a Hilbert space, the compactness requirement on the domain can be relaxed to closedness for ensuring the existence of fixed points.
\begin{theorem}\label{T35}
Let $K$ be a closed, bounded, convex subset of the Hilbert space $H$ and  $T: K \to K$ be a perimetric nonexpansive mapping. Then $T$ has a fixed point.
\end{theorem}

\begin{proof}
Let $\{s_n\}$ be a sequence of positive real numbers satisfying $s_n < 1$ for each $n$, with the property that $\lim_{n \to \infty} s_n = 1$.
Let $x_0 \in K$  be fixed. Then since $K$ is convex, for each $n\in \mathbb{N}$, the mapping  $U_n: K \to K$ defined by 
\begin{align*}
U_n(x)= (1-s_n)x_0 +s_nTx, \text{ for all } x \in K
\end{align*} is well defined. Moreover, $U_n$ is a mapping contracting perimeter of triangles on $K$, and given that $K$ is closed, it follows that $U_n$ has a fixed point, say $u_n$ for each $n$.

Since $K$ is closed, convex and bounded in the Hilbert space $K$, it is weakly compact. Then we  get a subsequence $\{u_{n_k}\}$ of $\{u_n\}$ such that $\{u_{n_k}\}$ converges weakly to an element $p \in K$. Next, we show that $p$ is a fixed point of $T$.

For any  $u \in K$, we have 
\begin{align*}
{\lVert u_{n_k} - u \rVert}^2 
&= {\lVert u_{n_k} -p + p - u \rVert}^2\\
&= {\lVert u_{n_k} - p \rVert}^2 + {\lVert p - u \rVert}^2 + 2\langle u_{n_k} -p, p-u \rangle.
\end{align*}
Here $\langle u_{n_k} -p, p-u \rangle \to \theta$ as $k \to \infty$ as $u_{n_k} - p$ converges to $\theta$ weakly in $K$.
Then taking $u=Tp$, we have
\begin{align*}
\lim_{k \to \infty} ({\lVert u_{n_k} - Tp \rVert}^2 - {\lVert u_{n_k} - p \rVert}^2) = {\lVert Tp - p \rVert}^2.
\end{align*}
Since $s_{n_k} \to 1$ and $U_{n_k}u_{n_k}=u_{n_k}$, we have
\begin{align*}
Tu_{n_k} - u_{n_k} &= (s_{n_k} Tu_{n_k}+ (1-s_{n_k})x_0) - u_{n_k} + (1-s_{n_k})(Tu_{n_k} - x_0)\\
&=(U_{n_k}u_{n_k} -u_{n_k}) + (1-s_{n_k})(Tu_{n_k} - x_0)\\
&=\theta + (1- s_{n_k})(Tu_{n_k} - x_0). 
\end{align*}
Thus, we have $\lim_{k \to \infty} \lVert Tu_{n_k} - u_{n_k} \rVert =0$.
Finally, we obtain
\begin{align*}
&\lVert u_{n_k} - Tp \rVert \le \lVert u_{n_k} - Tu_{n_k} \rVert+ \lVert Tu_{n_k} - Tp \rVert + \lVert Tp - Tu_{n_{k+1}} \rVert+ \lVert Tu_{n_{k+1}} - Tu_{n_k} \rVert\\
\implies &\lVert u_{n_k} - Tp \rVert \le \lVert u_{n_k} - Tu_{n_k} \rVert+ \lVert u_{n_k} - p \rVert + \lVert p - u_{n_{k+1}} \rVert+ \lVert u_{n_{k+1}} - u_{n_k} \rVert\\
\implies & \lim_{k \to \infty} \lVert u_{n_k} - Tp \rVert \le \lim_{k \to \infty} \lVert u_{n_k} - Tu_{n_k} \rVert =0\\
\implies & \lVert p - Tp \rVert =0\\
\implies &Tp=p.
\end{align*}
Hence, $p$ is a fixed point of $T$.
\end{proof}

\begin{Acknowledgement}
The first author of the paper would like to express gratitude to the Council of Scientific and Industrial Research, Government of India (File No: 09/0973(16599)/2023-EMR-I) for financially supporting in carrying out this work.  
\end{Acknowledgement}

\bibliographystyle{plain}

\end{document}